\documentclass{amsart}

\newtheorem{xxx}{XXX}[section]
\newtheorem{corollary}[xxx]{Corollary}
\newtheorem{proposition}[xxx]{Proposition}
\newtheorem{theorem}[xxx]{Theorem}
\newtheorem{lemma}[xxx]{Lemma}
\theoremstyle{definition}
\newtheorem{definition}[xxx]{Definition}
\newtheorem{example}[xxx]{Example}
\newtheorem{question}[xxx]{Question}

\newcommand{\REMARK}[1]{}
\newcommand{\compat}{\leftrightarrow}
\newcommand{\Nat}{{\mathbb N}}
\newcommand{\DIMEN}{\mathop{\mathrm{dim}}}
\usepackage[off]{auto-pst-pdf} 
\usepackage{graphicx}
\usepackage{psfrag}

\begin{document}

\title{Blocks of homogeneous effect algebras}
\author{Gejza Jen\v ca}

\begin{abstract} 
Effect algebras, introduced by Foulis and Bennett in
1994, are partial algebras which generalize some well known
classes of algebraic structures (for example orthomodular lattices, MV algebras,
orthoalgebras etc.). In the present paper, we introduce a new class of effect
algebras, called {\em homogeneous effect algebras}. This class includes
orthoalgebras, lattice ordered effect algebras and effect algebras satisfying
Riesz decomposition property. We prove that every homogeneous effect algebra is
a union of its blocks, which we define as maximal sub-effect algebras
satisfying Riesz decomposition property. This generalizes  a recent result by
Rie\v canov\'a, in which lattice ordered effect algebras
were considered. Moreover, the notion of a block of a homogeneous effect algebra
is a generalization of the notion of a block of an orthoalgebra. We prove that
the set of all sharp elements in a homogeneous effect algebra $E$ forms an
orthoalgebra $E_S$. Every block of $E_S$ is the center of a block of $E$.
The set of all sharp elements in the compatibility center of $E$ coincides 
with the center of $E$. 
Finally, we present some examples of homogeneous effect algebras
and we prove that for a Hilbert space $\mathbb H$ with $\DIMEN(\mathbb H)>1$,
the standard effect algebra
$\mathcal E(\mathbb H)$ of all effects in $\mathbb H$ is not homogeneous. 
\end{abstract} 

\address{Department of Mathematics, Faculty of Electrical Engineering and
Information Technology, Ilkovi\v cova~3, 812~19~Bratislava, Slovakia
}
\email{jenca@kmat.elf.stuba.sk}
\thanks{This research is supported by grant G-1/7625/20 of M\v S SR,
Slovakia}
\keywords{effect algebra, orthoalgebra, block, Riesz decomposition property}
\subjclass{Primary 06C15; Secondary 03G12}

\maketitle

\section{Introduction}

Effect algebras (or D-posets) have recently been introduced by Foulis and
Bennett in \cite{FouBen:EAaUQL} for study of foundations of quantum mechanics.
(See also \cite{KopCho:DP}, \cite{GiuGre:TaFLfUP}.)
The prototype effect
algebra is $(\mathcal E(\mathbb H),\oplus,0,I)$, where $\mathbb H$ is a
Hilbert space and $\mathcal E(\mathbb H)$ consists of all self-adjoint
operators $A$ of $\mathbb H$ such that $0\leq A\leq I$. For 
$A,B\in\mathcal E(\mathbb H)$, $A\oplus B$ is defined iff $A+B\leq 1$ and then
$A\oplus B=A+B$. $\mathcal E(\mathbb H)$ plays an important role in the
foundations of quantum mechanics \cite{Lud:FoQM}, \cite{BusGraLah:OQP}.

The class of effect algebras includes orthoalgebras \cite{FouGreRut:FaSiO} and
a subclass (called MV-effect algebras or Boolean D-posets or Boolean effect
algebras), which is essentially equivalent to MV-algebras, introduced by Chang
in \cite{Cha:AAoMVL} (cf. e.g. \cite{ChoKop:BDP}, \cite{BenFou:PSEA} for
results on MV-algebras in the context of effect algebras). The class of
orthoalgebras includes other classes of well-known sharp structures, like
orthomodular posets \cite{PtaPul:OSaQL} and orthomodular lattices
\cite{Kal:OL},\cite{Ber:OLaAA}.

One of the most important results in the theory of effect algebras was proved
by Rie\v canov\'a in her paper \cite{Rie:AGoBfLEA}. She proved that every
lattice ordered effect algebra is a union of maximal mutually compatible
sub-effect algebras, called blocks. This result generalizes the well-known fact
that an orthomodular lattice is a union of its maximal Boolean subalgebras.
Moreover, as proved in \cite{JenRie:OSEiLOEA}, in every lattice ordered effect
algebra $E$ the set of all sharp elements forms a sub-effect algebra $E_S$,
which is a sub-lattice of $E$; $E_S$ is then an orthomodular lattice, and every
block of $E_S$ is the center of some block of $E$. On the other hand, every
orthoalgebra is a union of maximal Boolean sub-orthoalgebras. Thus, although
the classes of lattice ordered effect algebras and orthoalgebras are
independent, both lattice ordered effect algebras and orthoalegebras are
covered by their blocks. This observation leads us to a natural question:

\begin{question}
Is there a class of effect algebras, say $\mathbb X$, with the following
properties?
\begin{itemize}
\item $\mathbb X$ includes orthoalgebras and lattice ordered effect algebras.
\item Every $E\in\mathbb X$ is a union of (some sort of) blocks.
\end{itemize}
\end{question}

In the present paper, we answer this question in the affirmative. We introduce
a new class of effect algebras, called homogeneous effect algebras. This class
includes lattice ordered effect algebras, orthoalgebras and effect algebras
satisfying Riesz decomposition property (cf. e.g. \cite{Rav:OaSToEA}). The
blocks in homogeneous algebras are maximal sub-effect algebras satisfying Riesz
decomposition property. We prove that the set of all sharp elements $E_S$ in a
homogeneous effect algebra $E$ forms a sub-effect algebra (of course, $E_S$ is an
orthoalgebra) and every block of $E_S$ is the center of a block of $E$. In the
last section we present some examples of homogeneous effect algebras and we
prove that $\mathcal E(\mathbb H)$ is not homogeneous unless $\DIMEN(\mathbb
H)\leq 1$.

\section{Definitions and basic relationships}

An {\em effect algebra} is a partial algebra $(E;\oplus,0,1)$ with a binary 
partial operation $\oplus$ and two nullary operations $0,1$ satisfying
the following conditions.
\begin{enumerate}
\item[(E1)]If $a\oplus b$ is defined, then $b\oplus a$ is defined and
		$a\oplus b=b\oplus a$.
\item[(E2)]If $a\oplus b$ and $(a\oplus b)\oplus c$ are defined, then
		$b\oplus c$ and $a\oplus(b\oplus c)$ are defined and
		$(a\oplus b)\oplus c=a\oplus(b\oplus c)$.
\item[(E3)]For every $a\in E$ there is a unique $a'\in E$ such that
		$a\oplus a'=1$.
\item[(E4)]If $a\oplus 1$ exists, then $a=0$
\end{enumerate}

Effect algebras were introduced by Foulis and Bennett in their paper 
\cite{FouBen:EAaUQL}. Independently, K\^ opka and Chovanec introduced
an essentially equivalent structure called {\em D-poset} (see \cite{KopCho:DP}).
Another equivalent structure, called {\em weak orthoalgebras} 
was introduced by Giuntini and Greuling in \cite{GiuGre:TaFLfUP}. 

For brevity, we denote the effect algebra $(E,\oplus,0,1)$ by $E$.
In an effect algebra $E$, we write $a\leq b$ iff there is $c\in E$ such
that $a\oplus c=b$.
It is easy to check that every effect algebra is cancellative, thus
$\leq$ is a partial order on $E$. In this partial order,
$0$ is the least and $1$ is the greatest element of $E$.
Moreover, it is possible to introduce
a new partial operation $\ominus$; $b\ominus a$ is defined iff
$a\leq b$ and then $a\oplus(b\ominus a)=b$.
It can be proved that $a\oplus b$ is defined iff $a\leq b'$ iff
$b\leq a'$. Therefore, it is usual to denote the domain of $\oplus$ by $\perp$.
If $a\perp b$, we say that $a$ and $b$ are {\em orthogonal}.
Let $E_0\subseteq E$ be such
that $1\in E_0$ and, for all $a,b\in E_0$ with $a\geq b$,
$a\ominus b\in E_0$. Since $a'=1\ominus a$ and $a\oplus b=(a'\ominus b)'$,
$E_0$ is closed with respect to $\oplus$ and $~'$.
We then say that $(E_0,\oplus,0,1)$ is a {\em sub-effect algebra of $E$}.
Another possibility to construct a substructure of an
effect algebra $E$ is to restrict $\oplus$ to an interval $[0,a]$,
where $a\in E$, letting $a$ act as the unit element. We denote such effect
algebra by $[0,a]_E$.

\noindent{\bf Remark. }For our purposes, it is natural to consider orthomodular
lattices, orthomodular posets, MV-algebras, and Boolean algebras as special
types of effect algebras. In the present paper, we will write shortly
``orthomodular lattice'' instead of ``effect algebra associated with an
orthomodular lattice'' and similarly for orthomodular posets, MV-algebras, and
Boolean algebras.

An effect algebra satisfying $a\perp a\implies a=0$ is called an {\em
orthoalgebra} (cf. \cite{FouGreRut:FaSiO}). An effect algebra $E$ is an {\em
orthomodular poset} iff, for all $a,b,c\in E$, $a\perp b\perp c\perp a$ implies
that $a\oplus b\oplus c$ exists (cf. \cite{FouBen:EAaUQL}). An orthoalgebra is
an {\em orthomodular lattice} iff it is lattice ordered.

Let $E$ be an effect algebra.
Let $C=(c_1,\ldots,c_n)$ be a finite family of elements of 
$E$. We say that $C$ is {\em orthogonal} iff the sum $c_1\oplus\ldots\oplus c_n$
exists. We then write $\bigoplus C=c_1\oplus\ldots\oplus c_n$. For $n=0$,
we put $\bigoplus C=0$. We say that $Ran(C)=\{c_1,\ldots,c_n\}$ is 
{\em the range of $C$}. 
Let $C=(c_1,\ldots,c_n),D=(d_1,\ldots,d_k)$ be orthogonal families of elements.
We say that $D$ is a {\em refinement of $C$} iff there is a partition
$P=\{P_1,\ldots,P_n\}$ of $\{1,\ldots,k\}$ such that, for all
$1\leq i\leq n$, $c_i=\bigoplus_{j\in P_i}d_j$. Note that if $D$ is a refinement of
$C$, then $\bigoplus C=\bigoplus D$.

A finite subset $M_F$ of an effect algebra $E$ is called
{\em compatible with cover in $X\subseteq E$} iff there is a finite
orthogonal family $C=(c_1,\ldots,c_n)$ with $Ran(C)\subseteq X$ such
that for every $a\in M_F$ there is a set $A\subseteq\{1,\ldots,n\}$ with
$a=\bigoplus_{i\in A}c_i$. $C$ is then called an {\em orthogonal
cover} of $M_F$.
A subset $M$ of $E$ is called {\em compatible with
covers in $X\subseteq E$} iff every finite subset of $M$ is compatible
with covers in $X$. A subset $M$ of $E$ is called {\em internally
compatible} iff $M$ is compatible with covers in $M$. A subset
$M$ of $E$ is called {\em compatible} iff $M$ is compatible with covers
in $E$. An effect algebra $E$ is said to be {\em compatible} if
$E$ is a compatible subset of $E$.
If $\{a,b\}$ is a compatible set, we write $a\compat b$.
It is easy to check that $a\compat b$ iff there are $a_1,b_1,c\in E$
such that $a_1\oplus c=a$, $b_1\oplus c=b$, and $a_1\oplus b_1\oplus c$ exists.
A subset $M$ of $E$ is called {\em mutually compatible} iff, for
all $a,b\in M$, $a\compat b$. Obviously, every compatible subset of an
effect algebra is mutually compatible. In the class of lattice ordered effect
algebras, the converse also holds. It is well known that in an 
orthomodular poset, a mutually compatible set need not to be compatible
(cf. e.g. \cite{PtaPul:OSaQL}).

A lattice ordered effect algebra $E$ is called an {\em MV-algebra} iff $E$ is
compatible (cf. \cite{ChoKop:BDP}). An MV-algebra which is an orthoalgebra is
a {\em Boolean algebra}. Recently, Z. Rie\v canov\'a proved in her paper
\cite{Rie:AGoBfLEA} that every lattice ordered effect algebra is a union of
MV-algebras, which are maximal mutually compatible subsets. These are called
{\em blocks}. She proved that every block of a lattice ordered effect algebra
$E$ is a sub-effect algebra and a sublattice of $E$. Note that Rie\v canov\'a's
results imply that every mutually compatible subset
of a lattice ordered effect algebra is compatible. Indeed, let $M$ be a
mutually compatible set. Then $M$ can be embedded into a block $B$, which is an
MV-algebra and hence compatible. Since $B$ is compatible and $M\subseteq B$,
$M$ is compatible.

On the other hand, it is easy to prove that every element of an
orthoalgebra can be embedded into a maximal sub-orthoalgebra, which is a Boolean
algebra. 

We say that an effect algebra $E$ satisfies {\em Riesz decomposition property}
iff, for all $u,v_1,\ldots,v_n\in E$ such that $v_1\oplus\ldots\oplus v_n$
exists and $u\leq v_1\oplus\ldots\oplus v_n$, there are $u_1,\ldots,u_n\in E$
such that,for all $1\leq i\leq n$, $u_i\leq v_i$ and
$u=u_1\oplus\ldots\oplus u_n$. It is easy to check that an effect algebra $E$
satisfies Riesz decomposition property iff $E$ satisfies Riesz decomposition
property with fixed $n=2$.
A lattice ordered effect algebra $E$ satisfies Riesz
decomposition property iff $E$ is an MV-algebra. An orthoalgebra $E$ satisfies
Riesz decomposition property iff $E$ is a Boolean algebra. 

Let $E_1,E_2$ be effect algebras. A map $\phi:E_1\mapsto E_2$ is called a {\em
morphism} iff $\phi(1)=1$ and $a\perp b$ implies that $\phi(a)\perp\phi(b)$ and
then $\phi(a\oplus b)=\phi(a)\oplus\phi(b)$. A morphism $\phi$ is an {\em
isomorphism} iff $\phi$ is bijective and $\phi^{-1}$ is a morphism.

\begin{definition}
\label{homogeneous}
An effect algebra $E$ is called {\em homogeneous} iff, for all $u,v_1,v_2\in E$
such that $v_1\perp v_2$, $u\leq v_1\oplus v_2$, $u\leq (v_1\oplus v_2)'$,
there are $u_1,u_2$ such that $u_1\leq v_1$, $u_2\leq v_2$ and
$u=u_1\oplus u_2$.
\end{definition}

\begin{proposition}~
\label{agoodclass}
\begin{enumerate}
\item[(a)]Every orthoalgebra is homogeneous.
\item[(b)]Every effect algebra satisfying Riesz decomposition property is
	homogeneous.
\item[(c)]Every lattice ordered effect algebra is homogeneous.
\end{enumerate}
\end{proposition}
\begin{proof}
For the proof of (a), observe that $u\leq v_1\oplus v_2$ and $u\leq (v_1\oplus
v_2)'$ imply that $u\perp u$ and thus $u=0$.
(b) is obvious. For the proof of (c), let $E$ be a lattice ordered
effect algebra. 
Note that $v_1\perp v_2$, $u\leq (v_1\oplus v_2)'$ imply that the set 
$\{u,v_1,v_2\}$ is mutually orthogonal and thus mutually compatible. Therefore,
by \cite{Rie:AGoBfLEA}, $\{u,v_1,v_2\}$ can be embedded into a block $B$. 
Since $B$ is an MV-algebra, $B$ satisfies Riesz decomposition property, 
hence $E$ is homogeneous. 
\end{proof}

\begin{proposition}
\label{homogeneousn}
Let $E$ be a homogeneous effect algebra. Let $u,v_1,\ldots,v_n\in E$
be such that $v_1\oplus\ldots\oplus v_n$ exists, 
$u\leq v_1\oplus\ldots\oplus v_n$ and
$u\leq (v_1\oplus\ldots\oplus v_n)'$. 
Then there are $v_1,\ldots,v_n$ such that, for all $1\leq i\leq n$,
$v_i\leq u_i$ and $u=u_1\oplus\ldots u_n$.
\end{proposition}
\begin{proof}
(By induction.) For $n=1$, it suffices to put $u_1=u$.
Assume that the proposition
holds for $n=k$. Let $u,v_1,\ldots,v_{k+1}$ be such that 
$v_1\oplus\ldots\oplus v_{k+1}$ exists,
$u\leq v_1\oplus v_2\oplus\ldots\oplus v_{k+1}$ and
$u\leq (v_1\oplus v_2\oplus\ldots\oplus v_{k+1})'$. Since $E$ is
homogeneous, there are $u_1\leq v_1$ and 
$z\leq v_2\oplus\ldots\oplus v_{k+1}$ such that $u=u_1\oplus z$.
Since
$$
z\leq u\leq (v_1\oplus\ldots\oplus v_{k+1})'
\leq (v_2\oplus\ldots\oplus v_{k+1})'
\text{,}
$$
we see that $z\leq(v_2\oplus\ldots\oplus v_{k+1})'$.
Thus, we may apply induction hypothesis. The rest is trivial.
\end{proof}

\section{Blocks of homogeneous effect algebras}

Let $E$ be an effect algebra. We say that a sub-effect algebra $B$ of $E$ is
a {\em block of $E$} iff $B$ is a maximal sub-effect algebra satisfying
the Riesz decomposition property. This definition of a block is consistent
with the definition of a block of the theory of orthoalgebras (maximal Boolean
sub-orthoalgebra) and also in the theory of lattice ordered effect algebras
(maximal mutually compatible subset).

In this section, we prove that blocks of homogeneous effect algebras coincide
with the maximal internally compatible subsets, which contain $1$.
As a consequence, every homogeneous effect algebra is a union of its blocks.

The main tool we use is the closure operation $M\mapsto\overline{M}$
which is defined on the system of all subsets of an effect algebra $E$ in the
following way. Let $M$ be a subset of an effect algebra $E$.
First we define certain subsets
$M_n$ ($n\in \Nat$) of $E$ as follows~: $M_0=M$ and for $n\in \Nat$
\begin{eqnarray}
\label{closure}
M_{n+1}&=&\{x:x\leq y,y'\text{ for some }y\in M_n\}\cup\\
\notag & &\{y\ominus x:x\leq y,y'\text{ for some }y\in M_n\}\text{.}
\end{eqnarray}
Then we put $\overline M=\bigcup_{n\in \Nat}M_n$. 
Note that, for all $n\in \Nat$, $M_n\subseteq M_{n+1}$ and that 
$\overline{\overline M}=\overline M$. In an orthoalgebra, $M=\overline M$ for
every set $M$.

\begin{lemma}
\label{maxcompat}
Let $E$ be an effect algebra. Let $M$ be an compatible subset of $E$. 
Then $M$ can be embedded into a maximal compatible
subset of $E$.
\end{lemma}
\begin{proof}
The proof is an easy application of Zorn's lemma and is left to the
reader.
\end{proof}

\begin{proposition}\label{crucial}
Let $E$ be a homogeneous effect algebra. Let $M\subseteq E$ be a
finite compatible set, $a,b\in M$, $a\geq b$. Let 
$C=(c_1,\ldots,c_k)$ be
an orthogonal cover of $M$. Let $A,B\subseteq \{1,\ldots,k\}$
be such that $a=\bigoplus_{i\in A}c_i$ and $b=\bigoplus_{i\in B} c_i$.
Then, there is a refinement of $C$,
say $W=(w_1,\ldots,w_n)$ and sets $B_W\subseteq A_W\subseteq\{1,\ldots,n\}$
such that
$(w_i)_{i\in A_W}$ is a refinement of $(c_i)_{i\in A}$ and
$(w_i)_{i\in B_W}$ is a refinement of $(c_i)_{i\in B}$. Moreover,
we have $Ran(W)\subseteq\overline{Ran(C_0)}$.

\end{proposition}
\begin{proof}
If $|B\setminus A|=0$ then $B\subseteq A$ and there is nothing
to prove.

Let $l\in\Nat$. Assume that Proposition \ref{crucial} holds for all
$C,A,B$ with $|B\setminus A|=l$.
Let $C_0,A_0,B_0$ be as in the assumption of Proposition \ref{crucial},
with $|B_0\setminus A_0|=l+1$. 

To avoid double indices, 
we may safely assume that $A_0$ and $B_0$ are such that,
for some $0\leq r,s,t\leq k$,
$B_0\setminus A_0=\{1,\ldots,r\}$,
$B_0\cap A_0=\{r+1,\ldots,s\}$, 
$A_0\setminus B_0=\{s+1,\ldots,t\}$. 

Write $b_1=c_1\oplus\ldots\oplus c_{l+1}$, 
$d=c_{l+2}\oplus\ldots\oplus c_s$, $a_1=c_{s+1}\oplus\ldots\oplus c_t$.
Since $b_1\oplus d=b\leq a=a_1\oplus d$, we see that $c_{l+1}\leq b_1\leq a_1$.
Since $C_0$ is an orthogonal family, $c_{l+1}\leq {a_1}'$.
By Proposition \ref{homogeneousn}, this implies that there are 
$v_{s+1},\ldots,v_t$ such that,
for all $s+1\leq i\leq t$, $v_i\leq c_i$ and 
$c_{l+1}=v_{s+1}\oplus\ldots\oplus v_t$.
Let us construct a refinement of $C_0$, say $C_1=(e_i)$, as follows.
\begin{eqnarray*}
C_1&=&(c_1,\ldots,c_l,
v_{s+1},\ldots,v_t,
c_{l+2},\ldots,c_s,\\
& &
c_{s+1}\ominus v_{s+1},\ldots,c_t\ominus v_t,
c_{t+1},\ldots,c_k,c_{l+1})
\end{eqnarray*}
Obviously, $C_1$ is a refinement of $C_0$ and 
$Ran(C_1)\subseteq\overline{Ran(C_0)}$. Moreover,
we have
\begin{center}
$
b=\bigoplus(c_1,\ldots,c_l,
v_{s+1},\ldots,v_t,
c_{l+2},\ldots,c_s)
$
\end{center}
and
\begin{center}
$
a=\bigoplus(v_{s+1},\ldots,v_t,
c_{l+2},\ldots,c_s,
c_{s+1}\ominus v_{s+1},\ldots,c_t\ominus v_t)\text{.}
$
\end{center}
By latter equations, we can find 
sets $A_1,B_1$ of indices such that $a=\bigoplus_{i\in A_1}e_i$,
$b=\bigoplus_{i\in B_1}e_i$ and
$B_1\setminus A_1=\{1,\ldots,l\}$. Moreover,
$(e_i)_{i\in A_1}$ is a refinement of $(c_i)_{i\in A_0}$ and
$(e_i)_{i\in B_1}$ is a refinement of $(c_i)_{i\in B_0}$.
As $|B_1\setminus A_1|=l$, we may apply the 
induction hypothesis on $C_1,A_1,B_1$ to find a refinement
$W=(w_1,\ldots,w_n)$ of $C_1$ with $Ran(W)\subseteq \overline{Ran(C_1)}$ 
and sets $B_W\subseteq A_W\subseteq\{1,\ldots,n\}$
such that $(w_i)_{i\in A_W}$ is a refinement of $(e_i)_{i\in A_1}$ and
$(w_i)_{i\in B_W}$ is a refinement of $(e_i)_{i\in B_1}$.
Obviously, $W$ is a refinement of $C_0$ and we see that
$$
Ran(W)\subseteq\overline{Ran(C_1)}\subseteq\overline{\overline{Ran(C_0)}}=
\overline{Ran(C_0)}\text{.}
$$
Similarly, $(w_i)_{i\in A_W}$ is a refinement of $(c_i)_{i\in A_0}$ and
$(w_i)_{i\in B_W}$ is a refinement of $(c_i)_{i\in B_0}$.
This concludes the proof. 
\end{proof}

\begin{corollary}
\label{minuscompat}
Let $M$ be a finite compatible subset of a homogeneous effect algebra $E$.
Let $a,b\in M$ be such that $a\geq b$. Then $M\cup\{a\ominus b\}$ is a
compatible set.
\end{corollary}
\begin{proof}
Let $W,A_W,B_W$ be as in Proposition \ref{crucial}.
Then $a\ominus b=\bigoplus_{i\in A_W\setminus B_W}w_i$, so
$W$ is an orthogonal cover of $M\cup\{a\ominus b\}$.
\end{proof}

\begin{corollary}
Let $M$ be a finite compatible subset of a homogeneous effect algebra $E$.
Let $a,b\in M$ be such that $a\perp b$. Then $M\cup\{a\oplus b\}$ is
a compatible set.
\end{corollary}
\begin{proof}
It is easy to check that, for every compatible set $M_0$, 
$M_0\cup {M_0}'=M_0\cup\{a':a\in M_0\}$ is a compatible set.
The rest follows from Corollary \ref{minuscompat} and
from the equation $a\oplus b=(a'\ominus b)'$.
\end{proof}

\begin{theorem}
Let $E$ be an effect algebra. The following are equivalent.
\label{equiv}
\begin{enumerate}
\item[(a)] $E$ satisfies Riesz decomposition property.
\item[(b)] $E$ is homogeneous and compatible.
\end{enumerate} 
\end{theorem}
\begin{proof}
(a) implies (b): It is evident that $E$ is homogeneous. It remains to
prove that every $n$-element subset of $E$ is compatible.
For $n=1$, there is nothing to prove. 
For $n>1$, let us assume that every $(n-1)$-element subset of $E$ is
compatible. Let $X=\{x_1,\ldots,x_n\}$ be a subset of $E$.
By induction hypothesis, $X_0=\{x_1,\ldots,x_{n-1}\}$ is compatible.
Thus, there is an orthogonal cover of $X_0$, say
$C=(c_1,\ldots,c_k)$. Since $x_n\leq (\bigoplus C)\oplus(\bigoplus C)'$ and
$E$ satisfies Riesz decomposition property, there exist
$y_1,y_2$ such that $y_1\leq (\bigoplus C)$, $y_2\leq(\bigoplus C)'$ and
$x_n=y_1\oplus y_2$. Since $y_1\leq (\bigoplus C)$, there
are $z_1,\ldots,z_k$ such that, for all $1\leq i\leq k$,
$z_i\leq c_i$ and $y_1=z_1\oplus\ldots\oplus z_k$.
Consequently, 
$$
(z_1,c_1\ominus z_1,\ldots,z_k,c_k\ominus z_k,y_2)
$$
is an orthogonal cover of $X$ and $X$ is compatible.

(b) implies (a): Let $u,v_1,v_2\in E$ be such that 
$v_1\perp v_2$, $u\leq v_1\oplus v_2$. If $v_1=0$ or $v_2=0$, there is
nothing to prove. Thus, let us assume that $v_1,v_2\not =0$.
By Proposition \ref{crucial}, $v_1\leq v_1\oplus v_2$ implies that 
there an orthogonal cover
$W=(w_1,\ldots,w_m)$ of $\{u,v_1,v_2,v_1\oplus v_2\}$ such that, for some 
$V_1\subseteq V\subseteq \{1,\ldots,m\}$, we have
$\bigoplus_{i\in V}w_i=v_1\oplus v_2$ and $\bigoplus_{i\in V_1}w_i=v_1$.
This implies that $\bigoplus_{i\in V\setminus V_1}w_i=v_2$.
By Proposition \ref{crucial}, $u\leq v_1\oplus v_2$ implies that 
there is a refinement of $W$, say $Q=(q_1,\ldots,q_n)$, and some
$U\subseteq Z\subseteq \{1,\ldots,n\}$ such that
$\bigoplus_{i\in U}q_i=u$ and $\bigoplus_{i\in Z}q_i=v_1\oplus v_2$.
Moreover, by Proposition \ref{crucial}, we may assume that
$(q_i)_{i\in Z}$ is a refinement of $(w_i)_{i\in V}$.
This implies that there is $Z_1\subseteq Z$ such that
$\bigoplus_{i\in Z_1}q_i=v_1$.
Put $u_1=\bigoplus_{i\in U\cap Z_1}q_i$ and
$u_2=\bigoplus_{i\in U\cap(Z\setminus Z_1)}q_i$.
It remains to observe that 
$u=u_1\oplus u_2$, $u_1\leq v_1$ and $u_2\leq v_2$.
\end{proof}

\begin{example}
Let $R_6$ be a 6-elements effect algebra with two atoms $\{a,b\}$,
satisfying equation 
$a\oplus a\oplus a=a\oplus b\oplus b=1$. Since
$(a,b,b)$ is an orthogonal cover of $R_6$, $R_6$ is a compatible
effect algebra. However, $R_6$ does not satisfy Riesz decomposition
property, since $a\leq b\oplus b$ and $a\land b=0$.
This example shows that there are compatible effect algebras that do
not satisfy Riesz decomposition property. 
\end{example}

\begin{proposition}
\label{intclosure}
Let $M$ be a subset of a homogeneous effect algebra $E$ such that
$M$ is compatible with covers in $\overline M$. Then $\overline M$ is
internally compatible.
\end{proposition}
\begin{proof}
Consider (\ref{closure}). Since each finite subset of $\overline M$ can be
embedded into some $M_n$, it suffices to 
prove that, for all $n\in\Nat$, $M_n$ is compatible with covers
in $\overline M$. By assumption, $M=M_0$ is compatible with covers in
$\overline M$. Assume that, for some $n\in \Nat$, $M_n$ is compatible with covers
in $\overline M$.
Evidently, every finite subset of $M_{n+1}$ can be embedded into a set of the
form 
\begin{equation}
\label{jedna}
\{x_1,y_1\ominus x_1,\ldots,x_k,y_k\ominus x_k\}\subseteq M_{n+1}\text{,}
\end{equation} where
for all $1\leq i\leq k$ we have $x_i\leq y_i,{y_i}'$ and $y_i\in M_n$.
We now prove the following

\noindent{\em Claim.} Let $x_i,y_i$ be as above.
For every cover $C_0$ of $\{y_1,\ldots,y_k\}$, there
is a refinement $W$ of $C_0$ such that $W$ covers 
$\{x_1,y_1\ominus x_1,\ldots,x_k,y_k\ominus x_k\}$ and
$Ran(W)\subseteq\overline{Ran(C_0)}$.

\noindent{\em Proof of the Claim.} For $k=0$, we may put $W=C_0$. Assume that
the Claim is satisfied for some $k=l\in \Nat$. Let $C_0$ be a cover of
$\{y_1,\ldots,y_{l+1}\}\subseteq M_{n}$. Since $C_0$ is a cover of
$\{y_1,\ldots,y_l\}$ as well, by induction hypothesis there is a refinement of
$C_0$, say $C_1$, such that $C_1$ covers $\{x_1,y_1\ominus
x_1,\ldots,x_l,y_l\ominus x_l\}$ and $Ran(C_1)\subseteq\overline{Ran(C_0)}$.
As $C_1$ is a refinement of $C_0$, $C_1$ covers $\{y_1,\ldots,y_{l+1}\}$.
Thus, there are $(c_1,\ldots,c_m)\subseteq C_1$ such that
$y_{l+1}=c_1\oplus\ldots\oplus c_m$. Since $x_{l+1}\leq y_{l+1},{y_{l+1}}'$,
Proposition \ref{homogeneousn} implies that there are $z_1,\ldots,z_m$ such
that, for all $1\leq i\leq m$, $z_i\leq c_i$ and $x_{l+1}=z_1\oplus\ldots\oplus
z_l$. Let us construct a refinement $W$ of $C_1$ by replacing each of the
$c_i$'s by the pair $(z_i,c_i\ominus z_i)$. Then $W$ is a refinement of $C_1$
and $W$ covers $\{x_1,y_1\ominus x_1,\ldots,x_{l+1},y_{l+1}\ominus x_{l+1}\}$.
Moreover, for all $1\leq i\leq m$, 
$z_i\leq x_{l+1}\leq {y_{l+1}}'\leq {c_i}'$, hence 
$$
Ran(W)\subseteq\overline{Ran(C_1)}\subseteq\overline{\overline{Ran(C_0)}}=
\overline{Ran(C_0)}\text {.}
$$

Now, let $M_F$ be a finite subset of $M_{n+1}$. We may assume that
$M_F$ is of the form (\ref{jedna}). By the outer induction hypothesis, $M_n$ is
compatible with covers in $\overline M$, thus
$\{y_1,\ldots,y_k\}$ is compatible with cover in $\overline M$. Let
$C$ be an orthogonal cover of $\{y_1,\ldots,y_k\}$ with
$Ran(C)\subseteq\overline M$. By Claim, there is a refinement $W$ of
$C$, such that $W$ covers $M_F$ and 
$Ran(W)\subseteq\overline{Ran(C)}\subseteq\overline{\overline M}=\overline M$.
Thus, $M_F$ is compatible with covers in $\overline M$ and we see that
$\overline M$ is internally compatible.
\end{proof}

The following are immediate consequences of Proposition \ref{intclosure}.

\begin{corollary}~
\label{theone}
\begin{enumerate}
\item[(a)]
Let $M$ be an internally compatible subset of a homogeneous effect algebra
$E$. Then $\overline M$ is an internally compatible set.
\item[(b)]
Let $M$ be a maximal internally compatible subset of a homogeneous
effect algebra $E$. Then $M=\overline M$.
\end{enumerate}
\end{corollary}

\begin{proposition}
\label{thetwo}
Let $E$ be a homogeneous effect algebra, let $M$ be an internally compatible
set with $M=\overline M$. Let $a,b\in M$, $a\geq b$. Then
$M\cup\{a\ominus b\}$ is an internally compatible set.
\end{proposition}
\begin{proof}
Let $M_F$ be a finite subset of $M$. Since $M$ is internally compatible,
there is an orthogonal cover $C$ of $M_F\cup\{a,b\}$ with $Ran(C)\subseteq M$.
By Corollary \ref{minuscompat}, $M_F\cup\{a,b,a\ominus b\}$ is then compatible
with cover in $\overline{Ran(C)}$. Therefore, $M_F\cup\{a\ominus b\}$
is compatible with cover in $\overline{Ran(C)}$.
Since $\overline{Ran(C)}\subseteq\overline M=M$, $M\cup\{a\ominus b\}$ is
an internally compatible set.
\end{proof}

As we will show later in Example \ref{lastone}, a sub-effect algebra of a
homogeneous effect algebra need not to be homogeneous. However,
we have the following relationship on the positive side.

\begin{proposition}
\label{subalg}
Let $E$ be a homogeneous effect algebra. Let $F$ be a sub-effect algebra of $E$
such that $F=\overline F$, where the closure is taken in $E$.
Then $F$ is homogeneous.
\end{proposition}
\begin{proof}
Let $u,v_1,v_2\in F$
be such that $u\leq v_1\oplus v_2$ and $u\leq (v_1\oplus v_2)'$. Since
$E$ is homogeneous, there are $u_1,u_2\in E$ such that $u_1\leq v_1$,
$u_2\leq v_2$ and $u=u_1\oplus u_2$. For $i\in\{1,2\}$, we have
$u_i\leq v_1\oplus v_2$ and $u_i\leq(v_1\oplus v_2)'$. Thus,
$u_1,u_2\in\overline F=F$ and $F$ is homogeneous.
\end{proof}

\begin{theorem}
\label{maxcompatisblock}
Let $E$ be a homogeneous effect algebra, let $B\subseteq E$. 
The following are equivalent.
\begin{enumerate}
\item[(a)]$B$ is a maximal internally compatible set with $1\in B$.
\item[(b)]$B$ is a block.
\end{enumerate}
\end{theorem}
\begin{proof}
Assume that (a) is satisfied. By Corollary \ref{theone}, part (b),
$B=\overline B$. By Proposition \ref{thetwo}, this implies that 
for all $a,b\in B$ such that $a\geq b$, $B\cup\{a\ominus b\}$ is an internally
compatible set. Therefore, by maximality of $B$, $B$ is closed with respect
to $\ominus$. Since $1\in B$, $B$ is a sub-effect algebra of $E$.
Since $B$ is an internally compatible set, 
$B$ is a compatible effect algebra. By Corollary \ref{theone}(b),
$B=\overline B$. By Proposition \ref{subalg}, this implies that
$B$ is homogeneous. Since $B$ is homogeneous and compatible,
Theorem \ref{equiv} implies that $B$ satisfies Riesz decomposition
property.

Assume that (b) is satisfied. By Theorem \ref{equiv},
$B$ is an internally compatible subset. By Lemma \ref{maxcompat},
$B$ can be embedded into a maximal internally compatible subset $B_{max}$ 
of $E$. By above part of the proof, $1\in B\subseteq B_{max}$ implies that 
$B_{max}$ is a block. Therefore, $B=B_{max}$ and (a) is satisfied.
\end{proof}

\begin{corollary}
\label{embedfinite}
Let $E$ be a homogeneous effect algebra. Every finite compatible subset of 
$E$ can be embedded into a block.
\end{corollary}
\begin{proof}
Let $M_F$ be a finite compatible subset of $E$. Let 
$C=(c_1,\ldots,c_n)$ be an orthogonal
cover of $M_F$. Then $M_F\cup\{1\}$ is compatible set, with cover 
$C^+=(c_1,\ldots,c_n,(\bigoplus C)')$. Thus, $M_F\cup\{1\}\cup Ran(C^+)$ is an
internally compatible set containing $1$. Therefore, by Lemma \ref{maxcompat},
$M_F\cup\{1\}\cup Ran(C^+)$ can be embedded into a maximal compatible
subset $B$ with $1\in B$. By Theorem \ref{maxcompatisblock},
$B$ is a block.
\end{proof}

\begin{corollary}
\label{blockcover}
Let $E$ be a homogeneous effect algebra. Then
$$
E=\cup\{B:B\text{ is a block of }E\}\text{.}
$$
\end{corollary}
\begin{proof}
By Corollary \ref{embedfinite}.
\end{proof}

\begin{corollary}
\label{bigcor}
For an effect algebra $E$, the following are equivalent.
\begin{enumerate}
\item[(a)]$E$ is homogeneous.
\item[(b)]Every finite compatible subset can be embedded into a block.
\item[(c)]Every finite compatible subset can be embedded into a sub-effect
	algebra of $E$ satisfying Riesz decomposition property.
\item[(d)]The range of every finite orthogonal family can be embedded into
	a block.
\item[(e)]The range of every finite orthogonal family can be embedded into 
	a sub-effect algebra satisfying Riesz decomposition property.
\item[(f)]The range of every orthogonal family with three elements can be 
	embedded into a block.
\item[(g)]The range of every orthogonal family with three elements can be 		embedded into a sub-effect algebra satisfying Riesz 
	decomposition property.
\end{enumerate}
\end{corollary}
\begin{proof}
(a)$\implies$(b) is Corollary \ref{embedfinite}. The implication chains
(b)$\implies$(c)$\implies$(e)$\implies$(g) and
(b)$\implies$(d)$\implies$(f)$\implies$(g) are obvious. To prove that
(h)$\implies$(a), assume that $E$ is an effect algebra satisfying (g), and
let $u,v_1,v_2\in E$ be such that
$u\leq v_1\oplus v_2$, $u\leq(v_1\oplus v_2)'$. Then
$(u,v_1,v_2)$ is an orthogonal family with three elements. By (g), 
$\{u,v_1,v_2\}$ can be embedded into a sub-effect algebra $R$ satisfying Riesz 
decomposition property. Thus, there are $u_1,u_2\in R\subseteq E$ such that
$u_1\leq v_1$, $u_2\leq v_2$ and $u=u_1\oplus u_2$. Hence,
$E$ is homogeneous.
\end{proof}

\begin{question}
Can every compatible subset of a homogeneous effect
algebra $E$ be embedded into a block ? This is true for orthomodular posets
(cf. e.g. \cite{PtaPul:OSaQL}) and for lattice ordered effect algebras.
By Theorem \ref{maxcompatisblock} and Lemma \ref{maxcompat},
this question reduces to the
question, whether a compatible subset can be embedded into an internally
compatible subset containing $1$.
\end{question}

\section{Compatibility center and sharp elements}

For a homogeneous effect algebra $E$, we write
$$
K(E)=\bigcap\{B:\text{$B$ is a block of $E$}\}\text{.}
$$
We say that $K(E)$ is the {\em compatibility center} of $E$.
Note that $K(E)=\overline{K(E)}$ and hence, 
by Proposition \ref{subalg}, $K(E)$ is homogeneous.

An element $a$ of an effect algebra is called {\em sharp} iff
$a\land a'=0$. We denote the set of all sharp elements of an effect algebra
$E$ by $E_S$. It is obvious that an effect algebra $E$ is an 
orthoalgebra iff $E=E_S$.
An element $a$ of an effect algebra $E$ is called {\em principal} iff
the interval $[0,a]$ is closed with respect to $\oplus$. Evidently, 
every principal element in an effect algebra is sharp.
A principal element $a$ of an effect algebra is called {\em central}
iff for all $b\in E$ there is a unique decomposition $b=b_1\oplus b_2$
with $b_1\leq a$, $b_2\leq a'$. The set of all central elements
of an effect algebra $E$ is called {\em the center of $E$} and 
is denoted by $C(E)$. 
In \cite{GreFouPul:TCoaEA}, the center of an effect algebra was introduced
and the following properties of $C(E)$ were proved.

\begin{proposition}
Let $E$ be an effect algebra. Then
\begin{itemize}
\item $C(E)$ is a sub-effect algebra of $E$.
\item $C(E)$ is a Boolean algebra. Moreover, for all $a\in C(E)$ and
	$x\in E$,
	$a\land x$ exists.
\item For all $a\in C(E)$, the map $\phi:E\mapsto[0,a]_E$ given by
	$\phi(x)=a\land x$ is a morphism.
\item For all $a\in C(E)$, $E$ is naturally isomorphic to 
	$[0,a]_E\times[0,a']_E$. Moreover, for all effect algebras $E_1,E_2$
	such that there is an isomorphism $\phi:E\mapsto E_1\times E_2$,
	$\phi^{-1}(1,0)$ and $\phi^{-1}(0,1)$ are central in $E$.
\end{itemize}
\end{proposition}

A subset $I$ of an effect algebra $E$ is called an {\em ideal}
iff the following condition is satisfied : $a,b\in I$, $a\perp b$ is
equivalent to $a\oplus b\in I$. An ideal $I$ is called {\em Riesz ideal}
iff, for all $i,a,b$ such that $i\in I$, $a\perp b$ and
$i\leq a\oplus b$, there are $i_1,i_2$ such that $i_1\leq a$, $i_2\leq b$ and
$i\leq i_1\oplus i_2$. Riesz ideals were introduced in \cite{GudPul:QoPAM}.

For a lattice ordered effect algebra $E$, it was proved in \cite{Rie:CaCEiEA},
that $C(E)=K(E)\cap E_S$. Moreover, as proved in \cite{JenRie:OSEiLOEA},
for a lattice ordered effect algebra $E$, $E_S$ is a sublattice of $E$,
a sub-effect algebra of $E$, and every block of $E_S$ is the center of a
block of $E$. In the remainder of this section, we will
extend some of these results to the class of homogeneous effect algebras.

\begin{proposition}
\label{blockcenter}
Let $a$ be an element of a homogeneous effect algebra $E$.
The following are equivalent.
\begin{enumerate}
\item[(a)]$a\in E_S$.
\item[(b)]$a$ is central in every block of $E$ which contains $a$.
\item[(c)]$a$ is central in some block of $E$.
\end{enumerate}
\end{proposition}
\begin{proof}
(a) implies (b): Assume that $a\in E$ is sharp, let $B$ be a block of $E$ 
such that $a\in B$. Since $a$ is sharp in $E$, $a$ is sharp in $B$.
We will prove that $a$ is principal in $B$. Let $x_1,x_2\in B$ be such that 
$x_1,x_2\leq a$, $x_1\perp x_2$. Since $B$ is a sub-effect algebra of $E$,
$x_1\oplus x_2\in B$. Since $B$ is internally compatible, 
$x_1\oplus x_2\compat a$ in $B$. 
By \cite{ChePul:SILiPAM}, Lemma 2, $x_1\oplus x_2\compat a$ in $B$
implies that there are $y_1,y_2\in B$ such that $y_1\leq a$, $y_2\leq a'$
and $x_1\oplus x_2=y_1\oplus y_2$. Since $B$ satisfies Riesz decomposition
property, $y_2\leq x_1\oplus x_2$ implies that there are $t_1,t_2\in B$
such that $t_1\leq x_1$, $t_2\leq x_2$ and $y_2=t_1\oplus t_2$.
For $i\in\{1,2\}$, $t_i\leq a,a'$. Since $a$ is sharp in $B$,
this implies that $t_1=t_2=0$. Thus, $x_1\oplus x_2=y_1\leq a$ and
$a$ is principal in $B$ and hence $[0,a]\cap B$ is an ideal in $B$.
Since $B$ satisfies Riesz decomposition property, every ideal in $B$ is
a Riesz ideal. By \cite{ChePul:SILiPAM}, an element $a$ of an effect algebra
is central iff $[0,a]$ is a Riesz ideal. Therefore, $a$ is central in $B$.

(b) implies (c):
By Corollary \ref{blockcover}, every element of $E$ is in some block.

(c) implies (a):
Let $a\in C(B)$ for some block $B$, let $b\leq a,a'$. Since
$B=\overline B$, $b\in B$. Thus, $b=0$ and $a$ is sharp.
\end{proof}

\begin{corollary}
\label{rieszcenter}
Let $a$ be an element of an effect algebra $E$ 
satisfying Riesz decomposition property. The following are equivalent.
\begin{enumerate}
\item[(a)]$a\in E_S$.
\item[(b)]$a\in C(E)$.
\item[(c)]$a$ is principal.
\end{enumerate}
\end{corollary}
\begin{proof}
By Proposition \ref{blockcenter}, (a) is equivalent to 
(b). In every effect algebra, all principal elements are sharp.
Every central element is principal.
\end{proof}

\begin{corollary}
For a homogeneous effect algebra $E$,
$E_S$ is a sub-effect algebra of $E$. Moreover, $E_S$ is an orthoalgebra.
\end{corollary}
\begin{proof}
Obviously, $0,1\in E_S$ and $E_S$ is closed with respect to $~'$. Assume $a,b\in E_S$,
$a\perp b$. Then $\{a,b\}$ is a finite compatible set.
Thus, by Corollary \ref{embedfinite},
$\{a,b\}$ can be embedded into a block $B$. 
By Proposition \ref{blockcenter}, $a,b\in C(B)$.
Since $C(B)$ is a sub-effect algebra of $B$, $a\oplus b\in C(B)$. 
By Proposition \ref{blockcenter}, $C(B)\subseteq E_S$, thus $a\oplus b\in E_S$.

Obviously, $E_S$ is an orthoalgebra.
\end{proof}

Since, for a homogeneous effect algebra $E$,
$E_S$ is an orthoalgebra, every compatible subset of $E_S$
can be embedded into a block of $E_S$, which is a Boolean algebra.

\begin{proposition}
Let $E$ be a homogeneous effect algebra. For every block 
$B^0$ in $E_S$ and for every block $B$ of $E$ such that 
$B^0\subseteq B$, $B^0=C(B)$.
\end{proposition}
\begin{proof}
Let $B^0$ be a block of $E_S$. Let $B$ be a block of $E$ with $B^0\subseteq B$.
By Proposition \ref{blockcenter}, $B^0\subseteq C(B)$.
Since $B^0$ is a block of $E_S$ and $C(B)$ is a Boolean algebra,
$B^0\subseteq C(B)$ implies that $B^0=C(B)$.
\end{proof}

\begin{question}
Let $B$ be a block of a homogeneous effect algebra $E$.
Is it true that $C(B)$ is a block of $E_S$?
\end{question}
 
\begin{proposition}
In a homogeneous effect algebra, $C(E)=C(K(E))=K(E)_S$.
\end{proposition}
\begin{proof}
It is evident that $C(E)\subseteq C(K(E))\subseteq K(E)_S$.
Let $a\in K(E)_S$. We shall prove that $[0,a]$ is a Riesz ideal.
By Lemma 2 of \cite{ChePul:SILiPAM}, this implies that $a\in C(E)$.
Suppose $x_1,x_2\leq a$, $x_1\perp x_2$. Then $\{x_1,x_2\}$ can be embedded
into a block $B$ of $E$. Since $a\in K(E)$, $a\in B$. Since $a$ is sharp,
$a$ is central in $B$. Thus, $a$ is principal in $B$ and 
hence $x_1\oplus x_2\leq a$. Therefore, $a$ is principal in $E$.
Let $i\in[0,a]$, $x\perp y$, $i\leq x\oplus y$. Similarly as above,
$\{a,x,y\}$ can be embedded into a block $B$ of $E$, such that
$a\in C(B)$. Obviously, $i\leq (x\oplus y)\land a$ and, since $a$ is
central in $B$, $(x\oplus y)\land a=(x\land a)\oplus(y\land a)$.
Thus, $[0,a]$ is a Riesz ideal. 
\end{proof}

\begin{question}
Let $E$ be a homogeneous effect algebra. Does $K(E)$
satisfy Riesz decomposition property ? This is true for orthoalgebras and
for lattice ordered effect algebras.
\end{question}

\section{Examples and counterexamples}

It is easy to check, that a direct product of a finite number of homogeneous
effect algebras is a homogeneous effect algebra.
\begin{example}
Let $E_1$ be an orthoalgebra.
Let $E_2$ be an effect algebra satisfying Riesz 
decomposition property, which is not an orthoalgebra.
If any of $E_1,E_2$ is not lattice ordered, then $E_1\times E_2$
is an example of a homogeneous effect algebra which is not lattice ordered.
Moreover, since $E_2$ is not an orthoalgebra, $E_1\times E_2$ is not an
orthoalgebra.
\end{example}

Another possibility to construct new homogeneous effect algebras from old
is to make {\em horizontal sums} (sometimes called {\em $0,1$-pastings)},
which means simply identifying the zeroes and ones of the summands.

As shown in the next example, it is possible to construct a lattice ordered
(and hence homogeneous) effect algebra by pasting of two MV-algebras
in a central element.

\begin{figure}
\psfrag{a}{$a$}
\psfrag{b}{$b$}
\psfrag{c}{$c$}
\psfrag{d}{$d$}
\psfrag{e}{$e$}
\includegraphics{box.eps}
\caption{~}
\label{box}
\end{figure}

\begin{figure}
\psfrag{0}{$0$}
\psfrag{1}{$1$}
\psfrag{a}{$a$}
\psfrag{g}{$a$}
\psfrag{b}{$b$}
\psfrag{c}{$c$}
\psfrag{d}{$d$}
\psfrag{e}{$e$}
\psfrag{a'}{$a'$}
\psfrag{g'}{$a'$}
\psfrag{b'}{$b'$}
\psfrag{c'}{$c'$}
\psfrag{d'}{$d'$}
\psfrag{e'}{$e'$}
\psfrag{c+c}{$c\oplus c$}
\psfrag{a+c}{$a\oplus c$}
\psfrag{b+c}{$b\oplus c$}
\psfrag{c+d}{$c\oplus d$}
\psfrag{c+e}{$c\oplus e$}
\psfrag{(c+c)'}{$(c\oplus c)'$}
\includegraphics{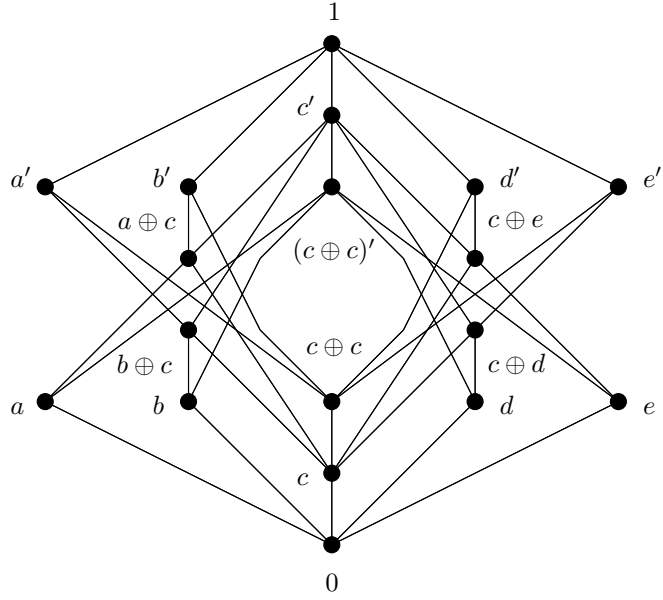}
\caption{An eighteen elements lattice ordered effect algebra}
\label{L18}
\end{figure}

\begin{example}
We borrowed the basic idea for this example from Cohen \cite{Coh:AItHSaQL}.
Consider a system consisting of a firefly in a box pictured in a
Figure \ref{box}. The box has five windows, separated by thin lines.
We shall consider two experiments on this system :

\begin{enumerate}
\item[(A)] Look at the windows $a,b,c$.
\item[(B)] Look at the windows $c,d,e$.
\end{enumerate}

Suppose that the window $c$ is covered with a grey filter. Unless the
firefly is shining very brightly at the moment we are performing the experiment,
we cannot be sure that we see the firefly in the $c$ window.
The outcomes of experiment (A) are
 
\medskip
\begin{tabular}{rl}
$(a)$&\mbox{We see the firefly in window $a$.}\\
$(b)$&\mbox{We see the firefly in window $b$.}\\
$(c)$&\mbox{We see the firefly in window $c$, with the level 
	of (un)certainity $\frac{1}{2}$.}\\
$(c\oplus c)$&\mbox{We see the firefly in window $c$.}
\end{tabular}
\medskip

The outcomes of (B) are similar. The unsharp quantum logic of our
experiment is an eighteen elements lattice ordered effect algebra $E$ with 
five atoms $a,b,c,d,e$, satisfying
$$
a\oplus b\oplus c\oplus c=c\oplus c\oplus d\oplus e=1\text{.}
$$
The Hasse diagram of $E$ is given by Figure \ref{L18}.
This effect algebra is constructed by pasting of two MV-algebras
$$
A=\{0,a,b,c,a\oplus c,b\oplus c,c\oplus c,a',b',(c\oplus c)',c',1\}
$$
and
$$
B=\{0,c,d,e,c\oplus c,c\oplus d,c\oplus e,d',e',(c\oplus c)',c',1\}\text{.}
$$
$A$ and $B$ are then blocks of $E$.
The compatibility center of $E$ is the MV-algebra
$$
K(E)=\{0,c,c\oplus c,(c\oplus c)',c',1\}
$$
and the center of $E$ is $\{0,c\oplus c,(c\oplus c)',1\}$.
$E_S$ forms a twelve-elements orthomodular lattice with two
blocks; each of them is isomorphic to the Boolean algebra $2^3$ and they are
pasted in one of their atoms (namely $c\oplus c$).
\end{example}

\begin{example}
Let $E$ be an eighteen elements effect algebra with six atoms
$a,b,c,d,e,f$, satisfying
\begin{equation}
\label{inmind}
a\oplus b\oplus c=c\oplus d\oplus d\oplus e=e\oplus f\oplus a=1\text{.}
\end{equation}
The Hasse diagram of $E$ is given by Figure \ref{gen}.
\begin{figure}
\psfrag{0}{$0$}
\psfrag{1}{$1$}
\psfrag{a}{$a$}
\psfrag{g}{$a$}
\psfrag{b}{$b$}
\psfrag{c}{$c$}
\psfrag{d}{$d$}
\psfrag{e}{$e$}
\psfrag{f}{$f$}
\psfrag{a'}{$a'$}
\psfrag{g'}{$a'$}
\psfrag{b'}{$b'$}
\psfrag{c'}{$c'$}
\psfrag{d'}{$d'$}
\psfrag{e'}{$e'$}
\psfrag{f'}{$f'$}
\psfrag{d+d}{$d\oplus d$}
\psfrag{d+e}{$d\oplus e$}
\psfrag{c+d}{$c\oplus d$}
\psfrag{(d+d)'}{$(d\oplus d)'$}
\includegraphics{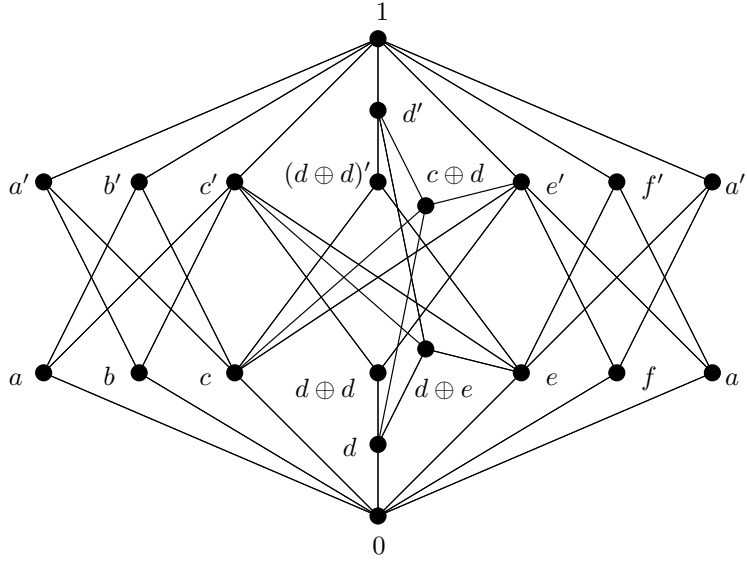}
\caption {A non-lattice ordered homogeneous effect algebra}
\label{gen}
\end{figure}
This effect algebra is constructed by pasting of three blocks :
two Boolean algebras
\begin{eqnarray*}
B_1&=&\{0,a,b,c,a',b',c',1\}\\
B_2&=&\{0,e,f,a,e',f',a',1\}
\end{eqnarray*}
and an MV-algebra
$$
B_3=\{0,c,d,e,d\oplus d,d\oplus e,c\oplus d,(d\oplus d)',c',d',e',1\}\text{.}
$$
By (\ref{inmind}), it is easy to see that the range of every 
orthogonal family with three elements can be embedded into a block. Thus, 
by Corollary \ref{bigcor}, $E$ is homogeneous.
All elements except for $d,d',c\oplus d,d\oplus e$ are sharp and
$E_S$ is an orthoalgebra with fourteen elements, called the 
{\em Wright triangle}, which is not an orthomodular poset.
\end{example}

\begin{proposition}
\label{nofunctions}
Let $E$ be a homogeneous effect algebra. Assume that there is an element
$a\in E$ with $a\leq a'$, such that $E$ is isomorphic to $[0,a]_E$.
Then $E$ satisfies Riesz decomposition property.
\end{proposition}
\begin{proof}
Let $B$ be a block containing $a$.
Since $B$ is a maximal internally compatible subset of $E$,
Corollary \ref{theone}(b) implies that 
$[0,a]=\{x\in E:x\leq a,a'\}\subseteq B$.
This implies that $[0,a]_E$ satisfies Riesz decomposition property.
Therefore, $E$ satisfies Riesz decomposition property.
\end{proof}

\begin{corollary}
For a Hilbert space $\mathbb H$,
$\mathcal E(\mathbb H)$ is homogeneous iff $\DIMEN(\mathbb H)\leq 1$. 
\end{corollary}
\begin{proof}
The map $\phi:\mathcal E(\mathbb H)\mapsto [0,\frac{1}{2}I]$ given by $\phi(A)=\frac{1}{2}A$ is
obviously an isomorphism and $\frac{1}{2}I\leq(\frac{1}{2}I)'$. Therefore,
by Proposition \ref{nofunctions}, every homogeneous  $\mathcal E(\mathbb H)$  
satisfies Riesz decomposition property.
However, it is well known that $\mathcal E(\mathbb H)$ satisfies Riesz decomposition property iff $\DIMEN(\mathbb H)\leq 1$.
\end{proof}

The following example shows that a sub-effect algebra of a homogeneous
effect algebra need not to be homogeneous.

\begin{example}\label{lastone}
Let $E=[0,1]\times[0,1]$, where $[0,1]\subseteq\mathbb R$ denotes the
unit interval of the real line. Equip $E$ with a partial operation 
$\oplus$ with domain given by $(a_1,a_2)\perp (b_1,b_2)$
iff $a_1+b_1\leq 1$ and $a_2+b_2\leq 1$; then define
$(a_1,a_2)\oplus (b_1,b_2)=(a_1+b_1,a_2+b_2)$.
Then $(E,\oplus_E,(0,0),(1,1))$ is a homogeneous effect algebra
(in fact, it is even an MV-algebra). Let 
$$
F=\{(x_1,x_2)\in E:x_1+x_2\in\mathbb Q\}
$$
Since $(1,1)\in F$ and $F$ is closed with respect to $\ominus$,
$F$ is a sub-effect algebra of $E$.

It is easy to see that the map 
$\phi:F\mapsto[(0,0),(\frac{1}{2},\frac{1}{2})]_F$,
given by $\phi(x_1,x_2)=(\frac{1}{2}x_1,\frac{1}{2}x_2)$ is an isomorphism.
Note that $F$ is not a compatible effect algebra: for example, 
$\{(1,0),(\frac{1}{\pi},1-\frac{1}{\pi})\}$ is not compatible in $F$.
Consequently, $F$ does not satisfy Riesz decomposition property and hence, 
by Proposition \ref{nofunctions}, $F$ is not homogeneous.
\end{example}

\begin{example}
Let $\mu$ be the Lebesgue measure on $[0,1]$.
Let $E\subseteq[0,1]^{[0,1]}$ be such that, for all $f\in E$,
\begin{enumerate}
\item[(a)]$f$ is measurable with respect to $\mu$
\item[(b)]$\mu(\mathrm{supp}(f))\in\mathbb Q$
\item[(c)]$\mu(\{x\in [0,1]: f(x)\not\in\{0,1\}\})=0$,
\end{enumerate}
where $\mathrm{supp}(f)$ denotes the support of $f$.
It is easy to check that $E$ is a sub-effect algebra of $[0,1]^{[0,1]}$.
Obviously, $E$ is not an orthoalgebra.
We will show that $E$ is a homogeneous, non-lattice ordered effect algebra and
that $E$ does not satisfy Riesz decomposition property.
Note that, for all $u\in E$, $u\perp u$ iff
$Ran(u)\subseteq[0,\frac 12]$ and $\mu(\mathrm{supp}(u))=0$.
Thus, for all $u\in E$ and $u_0\in [0,1]^{[0,1]}$ such that $u_0\leq u$ and 
$u\perp u$, we have $u_0\in E$.
 
Let $u,v_1,v_2\in E$ be such that $u\leq v_1\oplus v_2$, $u\leq(v_1\oplus
v_2)'$. Since $[0,1]^{[0,1]}$ is an MV-algebra, 
there are $u_1,u_2\in [0,1]^X$ such that $u_1\leq v_1$, 
$u_2\leq v_2$ and $u=u_1\oplus u_2$. By above paragraph,
$u\perp u$ and $u_1,u_2\leq u\in E$ imply that $u_1,u_2\in E$. Therefore,
$E$ is homogeneous.
Let $f,g$ be the characteristic functions of intervals $[0,\frac 23]$,
$[\frac {1}{\pi},\frac {1}{\pi}+\frac{1}{2}]$, respectively.
Then $f\land g$ does not exist in $E_S$. Therefore, $E_S$ is not lattice ordered
and hence, by Theorem 3.3 of \cite{JenRie:OSEiLOEA}, $E$ is not lattice ordered.
Moreover, $E$ does not satisfy Riesz decomposition
property. Indeed, assume the contrary.
Then, by Proposition \ref{rieszcenter}, $E_S=C(E)$. 
In particular, $E_S$ is then a Boolean algebra. However, 
this is a contradiction, since $E_S$ is not lattice ordered.
\end{example}

\providecommand{\bysame}{\leavevmode\hbox to3em{\hrulefill}\thinspace}

\end{document}